\documentclass[11pt]{article}
\usepackage{amsfonts,amssymb,amsmath,amscd,latexsym,makeidx,esint,amsthm,color,hyperref,graphicx,multicol,pdfsync}
\usepackage[utf8]{inputenc}
\usepackage[english]{babel}
\usepackage{comment}
\newtheorem{thm}{Theorem}[section]
\newtheorem{lem}{Lemma}[section]

\newtheorem{rem}{Remark}[section]
\newtheorem{prop}{Proposition}[section]
\newtheorem{defn}{Definition}[section]

\newcommand{\ov}{\overline}

\newcommand{\loc}{\mathrm{loc}}

\newcommand{\R}{\mathbb{R}}

\newcommand{\rn}{\mathbb{R}^{n}}

\newcommand{\eps}{\varepsilon}

\newcommand{\grad}{\nabla}
\newcommand{\De}{\Delta}
\newcommand{\ve}{\varepsilon}
\newcommand{\bra}[1]{\left(#1\right)}
\newcommand{\vp}{\varphi}

\hypersetup{
    colorlinks,
    citecolor=blue,
    linkcolor=blue,
    urlcolor=black
}

\title{Local and nonlocal singular Liouville equations in Euclidean spaces}
\author{Ali Hyder\thanks{The authors are supported by the Swiss National Science Foundation projects n. PP00P2-144669, PP00P2-170588/1 and P2BSP2-172064.}\\ {\small UBC Vancouver}\\ {\small \texttt{ali.hyder@math.ubc.ca} }\and Gabriele Mancini${}^*$\\  {\small Università di Padova}\\  {\small \texttt{gabriele.mancini@math.unipd.it} } \and Luca Martinazzi${}^*$ \\  {\small Università di Padova} \\ {\small \texttt{luca.martinzzi@math.unipd.it}}}

\begin{document}
\maketitle

\begin{abstract}
We study metrics of constant $Q$-curvature in the Euclidean space with a prescribed singularity at the origin, namely solutions to the equation
\begin{equation*}
(-\Delta)^\frac{n}{2}w=e^{nw}-c\delta_{0} \text{ on } \R^n,
\end{equation*}
under a finite volume condition. We analyze the asymptotic behaviour at infinity  and the existence of solutions for every $n\ge 3$ also in a supercritical regime. Finally, we state some open problems.
\end{abstract}

\section{Introduction}
In this paper we will deal with the classification of the solutions to the equation 
\begin{equation}\label{eq0}
(-\Delta)^\frac{n}{2}w=e^{nw}-c\delta_{0} \text{ on } \R^n,
\end{equation}
where $e^{n w}\in L^1(\R^n)$, $c\in\R$, $\delta_0$ is a Dirac mass at the origin and $n\ge 2$. Equations of this kind arise naturally in the study of the prescribed $Q$-curvature problem with singularities. We recall that the $Q$-curvature is a curvature of order $n$ that was introduced by \cite{Br,BO,FH} and in the last decades has been intensely studied in problems of conformal geometry. If $w$ is a solution of \eqref{eq0}, then $e^{2w}|dx|^2$ is a metric on $\R^n$ conformal to the Euclidean metric $|dx|^2$ and having $Q$-curvature equal to $1$ everywhere except at the origin, where it has a special kind of singularity. When $n = 2$, singularities of this kind are known as conical singularities and have been studied e.g. in \cite{Bar,PT,Tr}.

\medskip

Writing
$$ w=u+\frac{c}{\gamma_{n}}\log |x|,$$ 
we see that $u$ satisfies 
\begin{equation}\label{eq0b}
(-\Delta)^\frac{n}{2}u=|x|^{\frac{nc}{\gamma_{n}}}e^{nu} \text{ on } \rn,
\end{equation}
where $\gamma_n:=\frac{(n-1)!}{2}|S^{n}|$ is such that
$$(-\Delta)^\frac{n}{2}\left(\frac{1}{\gamma_n}\log \frac{1}{|x|}\right)=\delta_0\quad\text{ in }\rn,$$
and it will be convenient to study equation \eqref{eq0b}. More precisely, for a given $\alpha>-1$, we consider the  problem 
\begin{equation}\label{P0}
\begin{cases}
(-\Delta)^\frac{n}{2}u=|x|^{n\alpha}e^{nu} \text{ on } \rn,\\
\Lambda:=\int_{\rn}|x|^{n\alpha}e^{nu}\, dx <\infty.
\end{cases}
\end{equation}
Geometrically, if $u$ solves \eqref{P0}, then the metric $e^{2u}|dx|^2$ has $Q$-curvature $|x|^{n\alpha}$ and total $Q$-curvature $\Lambda$.
The parameter $\Lambda$ is also the volume of the metric $e^{2w}|dx|^2$, and it  plays a crucial role in the existence of solutions to \eqref{P0}.

In order to give a good definition of weak solutions to \eqref{P0}, we need to define $(-\Delta)^\frac n2$ for a sufficiently large class of functions. Let $\mathcal{S}(\R^n)$ be the Schwartz space of rapidly decreasing functions, and for $s>0$ set
$$L_s(\R^n):=\left\{w\in L^1_{\loc}(\R^n):\int_{\R^n}\frac{|w(x)|}{1+|x|^{n+2s}}dx<\infty\right\}.$$
Given a function $v\in L_s(\R^n)$, $s>0$, we define $(-\Delta)^s v$ as a tempered distribution as follows:
\begin{equation}\label{lapweak}
\langle(-\Delta)^s v, \varphi\rangle:=\int_{\rn}v(-\Delta)^s\varphi dx \qquad \text{for every }\varphi\in\mathcal{S}(\rn),
\end{equation}
where 
$$(-\Delta)^\frac{s}{2}\varphi:= \mathcal{F}^{-1}(|\xi|^{s}\mathcal{F}{\varphi}(\xi)),$$
and
$$\mathcal{F}\varphi(\xi):=\frac{1}{(2\pi)^\frac{n}{2}}\int_{\R^n}\varphi(x)e^{-ix\cdot\xi}dx$$ 
is the Fourier transform. The right-hand side of \eqref{lapweak} makes sense because
$$(-\Delta)^s\varphi(x)=O(|x|^{-n-2s})\quad \text{as }|x|\to\infty,$$
see e.g. \cite[Proposition 2.1]{H-clas}.

\begin{defn}\label{def-soln}
{Given $m\ge 1$ and a tempered distribution $f\in \mathcal{S}'(\R^n)$, we say that $u$ is a solution of $$(-\Delta)^\frac m2u=f\quad\text{in }\R^n$$
if the following holds. In case $m\ge 2$ even we require $u\in L^1_\loc(\R^n)$ and
$$\int_{\R^n}u(-\Delta)^{\frac{m}{2}}\vp dx=\langle f,\vp \rangle\quad\text{for every }\vp\in C^\infty_c(\R^n).$$ 
In case $m\ge 1$ odd we require $u\in L_{\frac m2}(\R^n)$ and
$$\int_{\R^n}u(-\Delta)^{\frac{m}{2}}\vp dx=\langle f,\vp \rangle\quad\text{for every }\vp\in \mathcal{S}(\R^n).$$ 
}
\end{defn}

As we shall see in Theorem \ref{thmreg}, weak solutions to \eqref{P0} are smooth away from the origin and at least $C^0$ near the origin.

In the classical case $n=2$, for any $\alpha>-1$ there are explicit solutions to \eqref{P0} taking the form
$$u(x)=\log\bra{\frac{2(\alpha+1)}{1+|x|^{2(\alpha+1)}} },\quad \Lambda=4\pi(\alpha+1).$$
With the scaling $u_\lambda(x)=u(\lambda^\frac{1}{1+\alpha} x)+\log\lambda$ we obtain the family of solutions
\begin{equation}\label{spherical2}
u(x)=\log\bra{\frac{2(\alpha+1)\lambda}{1+\lambda^2|x|^{2(\alpha+1)}} },\quad \Lambda=4\pi(\alpha+1),\quad \lambda>0,\alpha>-1.
\end{equation}
In fact, as shown by Chen-Li \cite{CL}, when $n=2$ for every solution to \eqref{P0} we necessarily have $\Lambda=4\pi(1+\alpha)$. In \cite{PT}  Prajapat-Tarantello proved that solutions as in \eqref{spherical2} exhaust the set of solutions of \eqref{P0} if $\alpha$ is not an integer, while when $\alpha\in\mathbb{N}$ instead there exist more solutions, all of which can be written using the complex notation in the form
\begin{equation}\label{spherical3}
u(z)=\log\bra{\frac{2(\alpha+1)\lambda}{1+\lambda^2|z^{\alpha+1} -\zeta|^2 }},\quad \lambda>0,\;\zeta\in\mathbb{C}.
\end{equation}
Interestingly the solutions in \eqref{spherical3} are not radially symmetric, except in the case $\zeta=0$ (when they reduce to \eqref{spherical2}) or $\alpha=0$ (in which case they are radial about the point $\zeta$) and blow up at the $\alpha+1$ roots of the complex equation $z^{\alpha+1}=\zeta$ as $\lambda\to \infty$.

When $n>2$, the structure of the set of solutions of \eqref{P0} becomes richer. If $\alpha=0$ one can still identify a special family of solutions, known as {\it normal} (cfr. Definition \ref{normal}) or {\it sperical} solutions, having an explicit expression similar to \eqref{spherical3} (with $\alpha=0$). These solutions have total $Q$-curvature $\Lambda=\Lambda_1:=(n-1)!|S^n|$ and are the only solutions of \eqref{P0} with $u=o(|x|^2)$ as $|x|\to +\infty$. However, in addition to these solutions there are also non-normal solutions, behaving polynomially at infinity and with $\Lambda$ not necessarily equal to $\Lambda_1$ see e.g. \cite{CC,H-odd,HM,JMMX,WY}.

When $n\ne 2$  and $\alpha\ne 0$ we do not have explicit formulas for solutions to \eqref{P0} (not even when $\alpha$ is integer, to the best of our knowledge) so that the existence of even one single solution is not obvious. We shall prove two existence results based on a Schauder fixed-point argument, in a spirit similar to that of previous works, such as \cite{H-odd,H-vol, HM, WY}. 

There are two crucial ingredients in this approach. One is to have a good ansatz, namely restrict the set among which we will look for solutions to functions having a particular asymptotic behaviour. The second ingredient is a precise information on the value of $\Lambda$ for normal solutions to \eqref{P0}, namely solutions to an integral equation, see Definition \ref{normal}. Both this properties are contained in the following theorem, which extends to the case $\alpha\ne 0$ the classification results in \cite{Lin,MarClass,WX,JMMX, H-clas}.


\begin{thm}\label{mainthm}
For $\alpha>-1$ let $u$ solve \eqref{P0} and define $v$ as follows
\begin{equation}\label{defv}
v(x)=\frac{1}{\gamma_n}\int_{\rn}\log \left (\frac{1+| y|}{|x-y|}\right )|y|^{n\alpha}e^{nu(y)}\, dy.
\end{equation}
Then there exists an upper-bounded polynomial $p$ of degree at most $n-1$ such that $u=v+p$. Moreover 
\begin{equation}\label{asympv}
\lim_{|x|\to\infty}\frac{v(x)}{\log|x|}=-\beta:=-\frac{\Lambda}{\gamma_n}.
\end{equation}
Finally, if $p$ is constant, then $\Lambda=\Lambda_1(1+\alpha)$.
\end{thm}

Notice that $p$ being upper bounded implies that $p$ has even degree, hence $p$ has degree at most $n-2$ when $n$ is even.

\begin{defn}[Normal solutions]\label{normal} We call a solution $u$ to \eqref{P0} normal if writing $u=v+p$ as in Theorem \ref{mainthm} the polynomial $p$ is constant. Equivalently if there exists a constant $c\in \R$ such that $u$ solves the integral equation
$$u(x)=\frac{1}{\gamma_n}\int_{\rn}\log \left (\frac{1+| y|}{|x-y|}\right )|y|^{n\alpha}e^{nu(y)}\, dy +c.$$
\end{defn}

The proof of Theorem \ref{mainthm} is innovative in comparison to the previous works, which it also simplifies considerably (see Remark \ref{rmkeasy}). It is based on a Campanato-space estimate (Lemma \ref{campanato2}). In fact, instead of having an $L^\infty$-upper bound on the function $v$, which is usually difficult to obtain in the case $\alpha=0$ and appears to be much more challenging in the case $\alpha\ne 0$, we content ourselves with a decay estimate on the oscillation of $v$, which will be sufficient to conclude.

We can summarize Theorem \ref{mainthm} by saying that all solutions to \eqref{P0} have the form $v+p$, where $v$ behaves logarithmically at infinity and $p$ is a polynomial, and if $p$ is constant, then $\Lambda=\Lambda_1(1+\alpha)$. Using this information we can move towards the existence results.

In the first result we show existence of radial solutions up to the critical treshold $\Lambda_1(1+\alpha)$. 

\begin{thm}\label{thm4}
Let $n\geq 3$ and $\alpha>-1$. Then for every $0<\Lambda< \Lambda_1(1+\alpha)$ there exists a radially symmetric solution to \eqref{P0}. For $\Lambda=\Lambda_1(1+\alpha)$ there exists a radially symmetric and normal solution to \eqref{P0}.
\end{thm}

Existence for every $\Lambda<\Lambda_1(1+\alpha)$ is obtained by a standard compactness argument. These solutions will have the form $v+p$ with $p(x)=-|x|^2$. The normal solutions obtained for the case $\Lambda=\Lambda_1(1+\alpha)$ are the higher-dimensional equivalent of \eqref{spherical2}, whose existence is not obvious, since we do not have explicit formulas for them. For this part we will use a blow-up argument together with a non-existence result based on a Pohozaev-type identity.


Next we show that breaking radial symmetry, we are able to produce solutions to \eqref{P0} above the critical threshold.

\begin{thm}\label{thm2}
Let $n\geq 3$. Then for every $\alpha>-1$ and for every $\Lambda\in (0,\Lambda_1)$ there exists a solution  to \eqref{P0}.
\end{thm}

Notice that $\Lambda_1> \Lambda_1(1+\alpha)$ for $\alpha\in (-1,0)$ in Theorem \ref{thm2}, hence we have an existence result above the critical threshold of compactness. 

As we shall see, the solutions given by Theorem \ref{thm2} are non-radial, by construction.
On the other hand, in dimension $3$ and $4$ this is also a necessity. Indeed from the Pohozaev identity we obtain:

\begin{prop}\label{poho}
For $n=3$ and $n=4$ every \emph{radial} solution $u$ to \eqref{P0} satisfies $\Lambda\le \Lambda_1(1+\alpha)$ with identity if and only if  $u$ is normal.
\end{prop}

\subsection{Some open problems}

Using a variational argument as in  \cite{H-odd} one should be able to find solutions to \eqref{P0} of the form $u=v+p$ with the polynomial $p$ prescribed. For instance one could try to prove that for $n\geq 3$, $\alpha>-1$, $0<\Lambda<\Lambda_1\min\{ 1,1+\alpha\}$ and a given polynomial $p$ with $deg(p)\leq n-1$ and satisfying \begin{align}\label{cond-p}
p(x)\xrightarrow{|x|\to\infty}-\infty,
\end{align}
there exists a solution $u$ to \eqref{P0} such that $$u(x)=p(x)-\frac{2\Lambda}{\Lambda_1}\log|x|+C+o(1), \quad o(1)\xrightarrow{|x|\to\infty}0.$$

Also the existence of solutions to \eqref{P0} for arbitrarily large $\Lambda$ is open, in analogy with the case $\alpha=0$ studied in \cite{HD,H-vol,LM-vol}. In the radial case, using methods from \cite{H-vol} it should be possible to prove that for $n\geq 5$, $\alpha>-1$ and for every $\Lambda>0$ there exists a radially symmetric solution to \eqref{P0}. 
Notice that the condition $n\ge 5$ is necessary in view of Proposition \ref{poho} and the known results in dimension $1$ and $2$. If instead we drop the radial symmetry, it is open whether already in dimension $3$ and $4$ we can have solutions to \eqref{P0} with arbitrarily large $\Lambda$, for $\alpha\neq 0$.


\paragraph{Notation} In what follows $B_R(x)$ will denote the ball of radius $R$ centered at $x$ (the dependence on $x$ will often be omitted if $x=0$) and $C$ will denote a generic constant that can change from line to line.

\paragraph{Acknowledgements} During his PhD at the University of Basel under the guidance of the third author, Stefano Iula collaborated with the first and second authors to the proof of Theorem \ref{mainthm}.

\section{Regularity of solutions}
If $\Omega\subseteq \R^n$ is an open set, we denote
$$C^s(\Omega):=\left\{u\in C^{\lfloor s\rfloor}(\Omega):D^{\lfloor s\rfloor}u\in C^{0,s-\lfloor s\rfloor}(\Omega)\right\},
 \quad s-1<\lfloor s\rfloor\leq s,\, \lfloor s\rfloor\in\mathbb{N},$$
Then, we set
$$C^s_{\loc}(\R^n):=\{u\in C^0(\R^n): u|_\Omega\in C^s(\Omega)\text{ for every }\Omega\Subset\R^n\}.$$

\begin{thm}\label{thmreg}
Let $u$ be a solution of \eqref{P0} with $\alpha>-1$. If $\alpha\in (-1,0)$ {or $n\alpha-1\in 2\mathbb N$}, then $u\in C^\infty(\R^{n}\setminus\{0\})\cap C^{s}_{\loc}(\R^n)$, for $s<n(1+\alpha)$. If $\alpha>0$ with $n\alpha\not\in {\mathbb{N}}$ we have $u\in C^\infty(\R^{n}\setminus\{0\})\cap C^{n(1+\alpha)}_{\loc}(\R^n)$, and if $n\alpha\in 2\mathbb{N}$ we have $u\in C^\infty(\R^n)$.
\end{thm}
\begin{proof}
 First, we claim that $e^{nu}\in L^q_{\loc}(\R^n)$ for any $q\ge 1$.  Indeed, given any $q$, we can take $\ve=\ve(q)$ such that $q <\frac{\gamma_n}{\ve}$ and we can split $|x|^{n\alpha}e^{nu}=f_1+f_2$, where $f_1,f_2\ge 0$ and
$$f_1\in L^1(\rn)\cap L^\infty(\rn),\quad \|f_2\|_{L^1(\R^n)}\le \ve.$$
Let us define the functions 
$$u_i(x):=\frac{1}{\gamma_n}\int_{\rn}\log\left(\frac{1+|y|}{|x-y|}\right)f_i(y)dy,\quad x\in\rn,\,i=1,2,$$
and  
$$
u_3:=u-u_1-u_2.
$$
 It is easy to see that $u_1\in C^{n-1}(\R^n)$ and that $u_3$ is $\frac{n}{2}$-harmonic and hence $u_3\in C^\infty(\R^n)$.  Moreover, using Jensen's inequality we get
  \begin{align}
   \int_{B_R}e^{nq u_2}dx&=\int_{B_R}\exp\left(\int_{\rn}\frac{nq\|f_2\|}{\gamma_n}\log\left(\frac{1+|y|}{|x-y|}\right)\frac{f_2(y)}{\|f_2\|}dy\right)dx\notag\\
   &\leq \int_{B_R}\int_{\rn}\exp\left(\frac{nq\|f_2\|}{\gamma_n}\log\left(\frac{1+|y|}{|x-y|}\right)\right)\frac{{f_2(y)}}{\|f_2\|}dydx\notag \\
   &=\frac{1}{\|f_2\|}\int_{\rn}{f_2(y)}\int_{B_R}\left(\frac{1+|y|}{|x-y|}\right)^\frac{nq\|f_2\|}{\gamma_n}dxdy\notag\\
   &\leq C(q,R),\label{br}
  \end{align}
  where  $\|\cdot\|$ denotes the  $L^1(\rn)$ norm. Hence the claim is proved.  
          
Set $\bar p:=-\tfrac{1}{\alpha}$ for $\alpha\in(-1,0)$ and $\bar p:=\infty$ if $\alpha>0$. Then $|\cdot|^{n\alpha}\in L^p_{\loc}(\rn)$ for every $1\le p <\bar p$. 
It follows that $(-\Delta)^\frac n2 u = |\cdot|^{n\alpha}e^{nu}\in L^p_{\loc}(\rn)$ for $p<\bar p$, and this implies $u\in W^{n,p}_{\loc}(\R^n)$ for $p<\bar p$. Indeed,  for any given $R>0$, we can write $u= v_1+v_2+v_3$, where $v_3$ is $\frac{n}{2}-$harmonic (and thus $v_3\in C^\infty(\R^n)$) and 
\begin{equation}\label{defvi}
v_i(x):=\frac{1}{\gamma_n}\int_{\rn}\log\left(\frac{1+|y|}{|x-y|}\right)g_i(y)dy,\quad x\in\rn,\,i=1,2,
\end{equation}
with
$$g_1= |\cdot|^{n\alpha}e^{nu} \chi_{B_{R}},\quad g_2= |\cdot|^{n\alpha}e^{nu} \chi_{B_{R}^c}.$$
Differentiating \eqref{defvi} we obtain that  $v_1\in W^{n,p}(B_\frac{R}{2})$ (by the Calderon-Zygmund theory) and $v_2\in C^\infty(B_\frac{R}{2})$.

By the Sobolev embedding we then infer that $u\in C^s_{\loc}(\R^n)$ for $s<n(1+ \min\{\alpha,0\})$.

Since ${|\cdot|}^{n\alpha}\in C^\infty(\R^n\setminus\{0\})$, by bootstrapping regularity we see that $u\in C^\infty(\R^n\setminus \{0\})$ for every $\alpha>-1$.
{Now, if $\alpha\ge 0$, we observe that  
$$
|\cdot|^{n\alpha}\in \begin{cases}
C^{n\alpha}(\R^n) & \text{ if }n\alpha \notin \mathbb N,\\
C^s_{\loc}(\R^n) & \text{ if } s<n\alpha,\; n\alpha -1\in 2\mathbb N,\\
C^\infty(\R^n) & \text{ if } n\alpha \in 2\mathbb N.
\end{cases}
$$ 
}
In any case, we can conclude the proof by bootstrapping regularity using Schauder estimates.
\end{proof}

\section{Proof of Theorem \ref{mainthm}}

\begin{lem}\label{below}
Let $u$ be a solution to Problem \eqref{P0} and $v$ as in \eqref{defv}.
Then for $|x|\geq 1$ we have
\begin{equation}\label{vbelow}
v(x)\geq  -\beta \log | x|,
\end{equation}
where $\beta$ is as in \eqref{asympv}. 
\end{lem}

\begin{proof}
Since $|x|\geq1$, thanks to the triangular inequality we have

\begin{equation}\notag
|x-y|\leq |x|+|y|\leq |x|+|x||y|=|x|\left(1+|y|\right),
\end{equation}
for any $y\in\rn$.
Therefore 
\begin{equation}\notag
\log\left(\frac{1+|y|}{|x-y|}\right)\geq -\log|x|.
\end{equation}
\end{proof}

\begin{lem}\label{u=v+p}
Let $u$ be a solution to Problem \eqref{P0} and $v$ as in \eqref{defv}.
Then $u=v+p$ where $p$ is a polynomial of degree at most $n-1$.
\end{lem}

\begin{proof}
Set $p=u-v$ so that $(-\Delta)^\frac{n}{2} p=0$. From Lemma \ref{below} we have
$$
p(x)\leq u(x)+\beta \log(1+ |x|)+C.
$$
Recalling that $|\cdot|^{n\alpha}e^{nu}\in L^1$, it follows from a Liouville-type theorem (see Theorem \ref{thm6} in the Appendix) that $p$ is a polynomial of degree at most $n-1$. 

%
\end{proof}


\begin{lem}
\label{polybound}
Let $p$ be a polynomial as in Lemma \ref{u=v+p}. Then
\begin{equation}\label{supp}
\sup_{x\in\rn}p(x)<+\infty.
\end{equation}
In particular $p$ has even degree. Moreover for every $q\ge 1$ and $0<\rho\leq  \rho_0$ there exists $C=C(q, \rho_0)$ such that
\begin{equation}\label{e^qp}
\int_{B_\rho(x)}e^{qp}dx\le C|x|^{n(\beta-\alpha)},
\end{equation}
for any $x\in \R^n$ with $|x|\ge 1$.
\end{lem}

\begin{proof}
We start by proving \eqref{supp}. Following \cite{MarClass} we define
$$f(r)=\sup_{\partial B_r} p.$$
From Theorem $3.1$ in \cite{gor} it follows that if $\sup_{\rn}p=+\infty$ then there exists $s>0$ such that
$$\lim_{r\to+\infty}\frac{f(r)}{r^s}=+\infty.$$
{From Lemma \ref{u=v+p} $p$ is a polynomial of degree at most $n-1$. In particular, we have that $|\grad p(x)|\leq C| x|^{n-2}$ for $|x|$ large.}
From Lemma \ref{below} it follows that there exists $R>0$ such that for every $r\geq R$ we can find $x_r$ with $|x_r|=r$ such that
$$u(y)=v(y)+p(y)\geq r^s$$
for $|y-x_r|\leq\frac{1}{r^{n-2}}$. Now using Fubini we get
\begin{equation*}
\begin{split}
\int_{\rn}|x|^{n\alpha}e^{nu}\, dx&\geq \int_R^{+\infty} \int_{\partial B_r(0)\cap B_{r^{2-n}}(x_r)}r^{n\alpha} e^{nr^s}\, d\sigma\, dr\\&\geq C\int_R^{+\infty}\frac{r^{n\alpha}e^{nr^s}}{r^{(n-2)(n-1)}}\, dr=+\infty,
\end{split}
\end{equation*}
that is a contradiction, hence \eqref{supp} is proven.

The proof of \eqref{e^qp} follows at once from \eqref{supp} and Lemma \ref{below}. Indeed, if $|x|\ge 2\rho_0$ then 
$$
C \ge   \int_{B_{\rho}(x)} |y|^{n\alpha}  e^{nu(y)} dy \ge  C \int_{B_{\rho}(x)} |y|^{n (\alpha-\beta)}  e^{np(y)} dy \ge C|x|^{n (\alpha-\beta) }  \int_{B_{\rho}(x)}  e^{q p(y)}dy,
$$
while \eqref{e^qp} is trivial for $|x|\in [1,2\rho_0)$. 
\end{proof} 

\begin{lem}\label{upv}
For any $\varepsilon>0$ there exists $R>0$ such that for $|x|\geq R$
\begin{equation}\label{equpv}
v(x)\leq \left(-\beta+\ve\right)\log| x|+\frac{1}{\gamma_n}\int_{B_1(x)}\log\left(\frac{1}{| x-y|}\right) |y|^{n\alpha}e^{nu} dy.
\end{equation}
\end{lem} 
\begin{proof}
The simple proof is similar to \cite[pag. 213]{Lin} and is omitted. 
\end{proof}

\begin{lem}\label{lpv}
For any $q\ge 1$, $\ve_1,\ve_2 >0$ there exists a constant $C=C(q,\ve_1,\ve_2)$ such that for $0<\rho \le 1$ and $x\in \R^{n}$
$$
\frac{1}{\rho^{n-\eps_2}} \int_{B_{\rho}(x)} e^{q v(y)} dy \le \frac{C}{|x|^{(\beta-\eps_1)q}}.
$$
\end{lem}
\begin{proof}
For $x$ in a compact set the statement is trivial. Set
$f(x)=|x|^{n\alpha}e^{nu(x)},$
and fix $R>0$ such that
$$\frac{q}{\gamma_n} \|f \|_{L^1(B_R^c)}\le \eps_2.$$
From Lemma \ref{upv} up to taking $R$ larger, we have for $|z|>R+1$
\begin{align*}
v(z) &\le (-\beta+\eps_1) \log|z| + \frac{1}{\gamma_n}\int_{B_1(z)}\log\left(\frac{1}{| z-y|}\right) f(y)\chi_{|z-y|\le 1 }dy\\
&\le (-\beta+\eps_1) \log|z| + \frac{1}{\gamma_n}\int_{B_R^c(0)}\log\left(\frac{1}{| z-y|}\right) f(y)\chi_{|z-y|\le 1 }dy.
\end{align*}
Applying Jensen inequality with respect to the measure $d\mu(y):= \frac{f(y)}{\|f\|_{L^1(B_R^c)}} dy$ in $B_R^c$, we get for $|x|>R+2$
\[
\begin{split}
\int_{B_{\rho}(x)} e^{q v(z)} dz & \le \int_{B_{\rho}(x)}  \frac{1}{|z|^{(\beta -\eps_1)q}} e^{ q \frac{\|f\|_{L^1(B_R^c)}}{\gamma_n}  \int_{B_R^c(0)}\log\left(\frac{1}{| z-y|}\right)  \chi_{|z-y|\le 1 } d\mu(y) } dz \\
&\le    \frac{C}{|x|^{(\beta -\eps_1)q}} \int_{B_{\rho}(x)}   \int_{B_R^c(0)} e^{ q \frac{\|f\|_{L^1(B_R^c)}}{\gamma_n}  \log\left(\frac{1}{| z-y|}\right)  \chi_{|z-y|\le 1 }  } d\mu(y) dz \\
& { \le}  \frac{C}{|x|^{(\beta -\eps_1)q}}  \int_{B_R^c(0)}  \int_{B_{\rho}(x)}  e^{ \eps_2 \log\left(\frac{1}{| z-y|}\right)  \chi_{|z-y|\le 1 }  } dz d\mu(y)\\
& \le     \frac{C}{|x|^{(\beta -\eps_1)q}}   \int_{B_R^c(0)}  \int_{B_{\rho}(x)}  \bra{ 1 + \frac{1}{|z-y|^{\eps_2}} } dz d\mu(y) \\
& \le \frac{C\rho^{n-\eps_2} }{|x|^{(\beta -\eps_1)q}}. 
\end{split}
\]
\end{proof}

\begin{rem}\label{rmkeasy} Using Lemma \ref{lpv} one gets a simpler proof of Theorem \ref{mainthm} in the classical case $\alpha=0$, or even if $\alpha\in (-1,\beta)$. Indeed using the H\"older inequality in \eqref{equpv}, with $\ve \le {\beta-\alpha}$ and applying Lemma \ref{lpv} with $\ve_1=\ve$ and $\rho=1$ we get for $|x|$ large
\[\begin{split}
\int_{B_1(x)}\log\left(\frac{1}{| x-y|}\right) |y|^{n\alpha}e^{nu} dy& {\le C |x|^{n\alpha}\int_{B_1(x)}\log\left(\frac{1}{| x-y|}\right)e^{nv} dy } \le C\frac{|x|^{{n\alpha}}}{|x|^{n(\beta-\ve)}}\le C.
\end{split}
\]
The proof of \eqref{asympv} follows at once from \eqref{equpv}, and the rest of the proof will follow easily from the Pohozaev identity as we shall see below.
\end{rem}

\begin{rem}\label{newrmk} Arguing as in Lemmas \ref{below}-\ref{lpv} one can also obtain the following: Let $v$ be a solution to $$v(x)=\frac{1}{\gamma_n}\int_{\R^n}\log\left(\frac{1+|y|}{|x-y|}\right)K(y)e^{nv(y)}dy + c $$ for some $c\in \R$ and some non-negative function $K\in L^\infty(\R^n\setminus B_1)$ with $Ke^{nv}\in L^1(\R^n)$. Then $$\lim_{|x|\to\infty}\frac{v(x)}{\log|x|}=-\frac{1}{\gamma_n}\int_{\R^n} Ke^{nv}dx.$$ \end{rem}

\medskip


From now on we shall assume $\alpha\ge \beta$.

\begin{lem}\label{campanato2}
We have
\begin{equation}\label{eqtau}
\tau(x):=\sup_{\rho\in (0,4]}\frac{1}{\rho^{n+\frac{1}{\log|x|}}}\int_{B_\rho(x)} \left|  v(y) - \fint_{B_\rho(x)} v(z)dz  \right| dy =o(1),
\end{equation}
with $o(1)\to 0$ as $|x|\to \infty$. 
As a consequence we have
\begin{equation}\label{vC0alpha}
[v]_{C^{0,\frac{1}{\log(|x_0|+1)}}(B_1(x_0))}:=\sup_{x,y\in B_1(x_0)\atop x\ne y}\frac{|v(x)-v(y)|}{|x-y|^{\frac{1}{\log(|x_0|+1)}}}=o(1)\log (|x_0|+1),
\end{equation}
with $o(1)\to 0$ as $|x_0|\to \infty$. 
 \end{lem}
 
\begin{proof} We start proving \eqref{eqtau}{.} We have
\[
\int_{B_{\rho}(x)}  \fint_{B_\rho(x)} \left|  v(y) - v(z)  \right|  dz dy 
 \le \int_{B_{\rho}(x)}  \fint_{B_\rho(x)}    \frac{1}{\gamma_n}\int_{\rn}  \left| \log\left(\frac{|z-\xi|}{|y-\xi |}\right)\right|  f(\xi )d\xi   dz dy, 
\]
where $f(x):=|x|^{n\alpha}e^{nu(x)}$.
By \eqref{e^qp}, Lemma \ref{lpv} and H\"older's inequality we get for given $\eps_1,\eps_2, r>0$, that
\[\begin{split}
\int_{B_r (x)} f(y) dy & \le C |x|^{n\alpha}   \bra{ \int_{B_{r}(x)} e^{2n v}dy}^{\frac{1}{2}}   \bra{ \int_{B_{r}(x)} e^{2n p}dy}^{\frac{1}{2}} \\
& \le  C |x|^{n\alpha}   \bra{ \frac{r^{n-\eps_2}}{|x|^{2n (\beta-\eps_1)}}}^\frac{1}{2} \bra{|x|^{n (\beta-\alpha)} }^\frac{1}{2} \\
& = C |x|^{c_1} r^{\frac{n}{2}-\frac{\eps_2}{2}}, \quad c_1:=\frac{n(\alpha-\beta)}{2}+n\ve_1.
\end{split}
\]
Choosing $r=2\sqrt \rho$, together with Lemma \ref{log} this yields
\begin{align*}
(I)&:= \int_{B_{\rho}(x)}  \fint_{B_\rho(x)}    \int_{B_{{2\sqrt{\rho}}}(x)}  \left| \log\left(\frac{|z-\xi|}{|y-\xi |}\right)\right|  f(\xi )d\xi   dz dy \\ 
&\leq C\rho^{n}\int_{B_{{2\sqrt{\rho}}}(x)}  f(\xi )d\xi \\
&\leq C\rho^{\frac{5n}{4}-\frac{\ve_2}{4}}|x|^{c_1}.
\end{align*}
Now choosing $\ve_2\le \frac{n}{4}$ and taking $|x|$ sufficiently large, we have {$\frac{n}{4}-\frac{\ve_2}{4}-\frac{1}{\log|x|}{>}\frac{n}{8}$.} Then,  for $\rho\in (0,|x|^{-\frac{8c_1}{n}})$, we further bound 
\begin{align*}
(I)&\le C\rho^{n+\frac{1}{\log|x|}}\rho^{\frac{n}{4}-\frac{\ve_2}{4}-\frac{1}{\log|x|}}|x|^{c_1}\\
&=o(1)\rho^{n+\frac{1}{\log |x|}}\rho^{\frac{n}{8}}|x|^{c_1}\\
&=o(1)\rho^{n+\frac{1}{\log |x|}},
\end{align*}
and with Lemma \ref{log} part $ii)$
\begin{align*}
(II)&:= \int_{B_{\rho}(x)}  \fint_{B_\rho(x)}    \int_{B_{2\sqrt{\rho}}(x)^c}  \left| \log\left(\frac{|z-\xi|}{|y-\xi |}\right)\right|  f(\xi )d\xi   dz dy 
\leq C\rho^{n+\frac12} =o(1) \rho^{n+\frac{1}{\log |x|}}.
\end{align*}
 For $\rho\in (|{x}|^{-\frac{8c_1}{n}}, 1)$  we write 
\begin{align*}
\int_{B_{\rho}(x)} \fint_{B_\rho(x)}    \int_{\mathbb R^{n}}  \left| \log\left(\frac{|z-\xi|}{|y-\xi |}\right)\right|  f(\xi )d\xi   dz dy  
 =I_1+I_2,
 \end{align*}
 where $$I_i:=\int_{B_\rho(x)} \fint_{B_\rho(x)}    \int_{A_i}  \left| \log\left(\frac{|z-\xi|}{|y-\xi |}\right)\right|  f(\xi )d\xi   dz dy,\quad A_1:=B_{\frac{|x|}{2}}(x),\,  A_2=A_1^c.$$
 Using  Lemma \ref{log} we bound 
{ 
$$I_1\leq C\rho^n  \int_{A_1}  f(\xi )d\xi=o(1)\rho^n=o(1)\rho^{n+\frac{1}{\log|x|}},$$  
and
 $$I_2=o(1)\rho^n=o(1)\rho^{n+\frac{1}{\log|x|}},$$
} 
where we have used that $$\|f\|_{L^1(A_1)}\xrightarrow{|x|\to\infty}0,\quad |{x}|^\frac{1}{\log |{x}|}=e, \quad {dist(A_2,x)\xrightarrow{|x|\to\infty}\infty}.$$
This proves \eqref{eqtau}.

\medskip

For the proof of \eqref{vC0alpha} we essentially follow Theorem 5.5  of \cite{Gia-Mar}. Given $x\in \R^n$ and $\rho>0$ we use the notation
$$v_{x,\rho}:=\fint_{B_\rho(x)}v(y)dy.$$
Fix
$$
\sigma=\frac{1}{\log(|x_0|+1)}\in(0,1) \quad  \text{ and }\quad \lambda = n+\sigma.$$
For $0<r<R\le 4$, $x\in B_1(x_0)$ and $z\in B_R(x)$ we have
$$|v_{x,r}-v_{x,R}| \leq |v(z)-v_{x,r}|+|v(z)-v_{x,R}|,$$
and integrating with respect to $z$ we bound
\[\begin{split}
|v_{x,R}-v_{x,r}| & \le \frac{1}{|B_r|} \bra{ \int_{B_r(x)} |v(z)-v_{x,r}| dz + \int_{B_R(x)} |v(z)-v_{x,R}|  dz } \\
&\le \frac{C}{r^n}\bra{ r^\lambda+R^\lambda} \tau(x)  \\
& \le  C  R^\lambda r^{-n}\tau(x).
\end{split}
\] 
Setting $R_k = \frac{R}{2^k}$ we infer
$$|v_{x,R_k}-v_{x,R_{k+1}}|\leq C  R^{\sigma} 2^{-k{\sigma}}\tau(x).$$
Applying the triangular inequality for $h>k$ we bound
$$
|v_{x,R_h}- v_{x,R_k}|\le \sum_{j=k}^{h-1} |v_{x,R_{j+1}}-v_{x,R_{j}}| \le C R^{{\sigma}} \sum_{j=k}^{h-1} 2^{-{j\sigma}} \tau(x) \le \frac{C R^{\sigma}}{1-2^{-{\sigma}}} \tau(x).
$$
Since the function $s\mapsto \frac{s}{1-2^{-s}}$ is increasing in $[0,1]$ one has $\frac{1}{1-2^{-{\sigma}}}\le \frac{2}{{\sigma}}$ and we get 
$$
|{v_{x,R_h}-v_{x,R_k}}|\le \frac{C}{{ \sigma}} R^{{ \sigma}} \tau(x) \qquad 0\le k <h.
$$
Taking { $k=0$} and letting { $h\to\infty$} we now obtain
\begin{equation}\label{stimacamp}
|v_{x,R}- v(x)|\le \frac{C}{{\sigma}} R^{\sigma} \tau(x),\quad x\in B_1(x_0).
\end{equation}

For $x, y \in B_1(x_0)$ with $x\ne y$, take $R = |x-y|$. Then with \eqref{stimacamp} and the triangle inequality we bound
\begin{equation}\label{semi1}
\begin{split}
|v(x)-v(y)| & \le |v_{x,2R}-v(x)| + |v_{x,2R}-v_{y,2R}| + |v_{y,2R}-v(y)| \\
& \le \frac{C R^{\sigma}(\tau(x)+\tau(y))}{{\sigma}} + |v_{x,2R}-v_{y,2R}|.
\end{split}
\end{equation}
For any $z\in \R^n$, we have 
$$
|v_{x,2R}-v_{y,2R}|\le |v_{x,2R}-v(z)|+|v_{y,2R}-v(z)|.
$$
Integrating as $z\in B_{2R}(x)\cap B_{2R}(y)$   we get
\[
\begin{split}
|v_{x,2R}-v_{y,2R}|&\le \frac{1}{|B_{2R}(x)\cap B_{2R}(y)|} \bra{\int_{B_{2R}(x)} |v(z)-v_{x,2R}| dz + \int_{B_{2R}(y)} |v(z)-v_{y,2R}| dz} \\
&  \le \frac{C R^{\lambda} (\tau(x)+\tau(y))}{R^n}\\
& \le C R^{{\sigma}}(\tau(x)+\tau(y)).
\end{split}
\]
From \eqref{eqtau} and \eqref{semi1} we finally infer
$$\frac{|v(x)-v(y)|}{|x-y|^{ \sigma}}\le \frac{C}{{ \sigma}}(\tau(x)+\tau(y))=o(1)\log(|x_0|+1).$$

\end{proof}



\medskip\noindent\textbf{Proof of Theorem \ref{mainthm}.}
{First we prove that \eqref{asympv} holds. } From Lemma \ref{below} we have  $$\liminf_{|x|\to\infty}\frac{v(x)}{\log |x|}\geq -\beta.$$ We assume by contradiction that $$\lim_{|x|\to\infty}\frac{v(x)}{\log |x|}\neq -\beta.$$
Then there exists a sequence of points $(x_k)$ in $\R^{n}$ such that $|x_k|\to\infty$ and 
$$v(x_k)\geq (-\beta+2\delta)\log|x_k|\quad\text{for some }\delta>0.$$
Indeed, for $|x_k|$ large,  by Lemma \ref{campanato2} $$v(x)= v(x_k)+o(1)\log |x_k|\geq (-\beta+\delta)\log|x_k|\quad\text{for }x\in B_1(x_k).$$
Hence $$\lim_{k\to\infty}|x_k|^{\beta-\delta}\int_{B_1(x_k)}e^{v(x)}dx\geq \lim_{k\to\infty}|x_k|^{\beta-\delta}\int_{B_1(x_k)}e^{(-\beta+\delta)\log|x_k|}dx=|B_1|.$$
This contradicts to Lemma \ref{lpv} with $\rho=1$, $q=1$ and $0<\ve_1<\delta$. {  Thus \eqref{asympv} is proved. }

It remains to show that $\Lambda=\Lambda_1(1+\alpha)$ if $p$ is constant. In this case we have
$$u(x)=\frac{1}{\gamma_n}\int_{\rn}\log \left (\frac{1+| y|}{|x-y|}\right )|y|^{n\alpha}e^{nu(y)} dy+C.$$ 
Then, we are in position to apply the Pohozaev-type identity of Proposition \ref{special-poho} to conclude that $\Lambda= \Lambda_1(1+\alpha)$.  

\hfill$\square$


\section{Proof of Theorem \ref{thm4}}

When $\Lambda\in (0,\Lambda_1(1+\alpha))$ we will look for solutions of the form $u(x)=v(x)-|x|^2+c$ where $c\in \R$ and $v$ satisfies the integral equation
\begin{equation}\label{eqvint}
v(x)=\frac{1}{\gamma_n}\int_{\R^n}\log\left(\frac{1}{|x-y|}\right)|y|^{n\alpha}e^{-n|y|^2}e^{n(v(y)+c)}dy,
\end{equation}
so that in particular
$$(-\Delta)^\frac n2 v(x) =|x|^{n\alpha}e^{-n|x|^2}e^{n(v(x)+c)}.$$
Our approach will be based on Schauder's fixed-point theorem  (see \cite[Theorem 11.3]{GT}), an idea already exploited in several works. More precisely we set
$$X:=\left\{v\in C^0_{rad}(\R^n):\|v\|_X<\infty\right\},\quad \|v\|_X:=\sup_{x\in\R^n}\frac{|v(x)|}{1+|x|}.$$ 
For $v\in X$ we set $c_v\in \R$ such that
$$\int_{\R^n}|y|^{n\alpha}e^{-n|y|^2}e^{n(v(y)+c_v)}dy=\Lambda,$$
and define $T=T_\Lambda:X\to X$, $ Tv=\bar v$ where
$$\bar v(x):=\frac{1}{\gamma_n}\int_{\R^n}\log\left(\frac{1}{|x-y|}\right)|y|^{n\alpha}e^{-n|y|^2}e^{n(v(y)+c_v)}dy.$$

\begin{lem}\label{lemmacomp0}
The operator $T:X \to X$ is compact. 
\end{lem}

\begin{proof} Continuity follows by dominated convergence. Let now $(v_k)\subset X$ be a bounded sequence. From the definition of $c_{v_k}$ it follows easily that $|c_{v_k}|\le C$.  Therefore, \begin{align}\label{estbarvk}|\bar v_k(x)|\leq C\int_{\R^n}|\log|x-y|||y|^{n\alpha}e^{-|y|^2}dy\leq C\log (2+|x|).\end{align} Moreover, $$|\bar v_k(x)-\bar v_k(z)|\le C \int_{\R^n}\left|\log\left(\frac{|z-y|}{|x-y|}\right)\right||y|^{n\alpha}e^{-|y|^2}dy \to 0,\quad \text{as }|x-z|\to 0,\;x,z\in \R^{n}.$$ Thus, the sequence $(\bar v_k)$ is equicontinuous on $\R^n$.  Hence, by the theorem of Ascoli-Arzel\`a, up to a subsequence, $\bar v_k\to v$  in $C^0_{\loc}(\R^n)$ for some $v\in C^0(\R^n)$.  In particular, $\bar v_k\to v$ in $X$, thanks to \eqref{estbarvk}.  
\end{proof}

\begin{lem}\label{lemmamono}
The function $\bar v$ is radially decreasing.
\end{lem}
\begin{proof} Consider the functions
$$\bar v_\ve(x):=\frac{1}{\gamma_n}\int_{\R^n\setminus B_\ve}\log\left(\frac{1}{|x-y|}\right)|y|^{n\alpha}e^{-n|y|^2}e^{n(v(y)+c_v)}dy.$$
Differentiating under the integral sign one gets $\De \bar v_\ve<0,$ which implies that $\bar v_\ve$ is radially decreasing. Letting now $\ve\to 0$
we get $\bar v_\ve\to \bar v$ by dominated convergence, hence $\bar v$ is radially decreasing.
\end{proof}

\begin{lem}\label{diffint}
There exist constants $C, C(s)>0$, for $s> 0$, such that 
\begin{equation}\label{integral bounds bar v}
\int_{B_R} |\Delta \ov v| dx \le C R^{n-2}, \quad \int_{B_R^c} \frac{|\Delta \bar v|}{|x|^{n+s}}dx \le \frac{C(s)}{R^{s+2}},
\end{equation} 
for any $v\in X$ and $R>0$. 
\end{lem}

\begin{proof}
For any $y\in \R^n$ we have
$$
\int_{B_R}\frac{1}{|x-y|^2} dx \le \int_{B_R(y)}\frac{1}{|x-y|^2}dx + \int_{B_R \cap B_R(y)^c} \frac{1}{|x-y|^2}dx \le C R^{n-2}.
$$
Then writing with Fubini-Tonelli's theorem
\begin{align*}
\int_{B_R}|\Delta\bar v|dx &\le  C\int_{\R^n}|y|^{n\alpha}e^{-n|y|^2}e^{n(v(y)+c_v)} \bigg(\int_{B_R} \frac{1}{|x-y|^2}dx \bigg)dy\\
&\le C R^{n-2}\Lambda
\end{align*}
the first estimate in \eqref{integral bounds bar v} follows. The second one is proven in the same way since 
$$
\int_{B_R^c} \frac{1}{|x-y|^2 |x|^{n+s}}dx \le  \int_{B_R(y)}\frac{1}{|x-y|^2 R^{n+s}}dx + \int_{B_R^c \cap B_R(y)^c} \frac{1}{R^2|x|^{n+s}}dx \le \frac{C}{R^{s+2}}.
$$
\end{proof}

We are now in a position to prove the main a priori estimate.
\begin{prop}\label{aprioriTv} Assume that $\Lambda\in (0,\Lambda_1(1+\alpha))$. Then there exists $C>0$ such that for every $(v,t)\in X\times [0,1]$  satisfying $v=tT(v)$ we have $\|v\|\leq C.$
\end{prop}

\begin{proof}
Assume for the sake of contradiction that there exists $(v_k,t_k)\in X\times [0,1]$ such that
$$v_k=t_kT(v_k)\quad\text{and } \|v_k\|\to\infty.$$
Then $v_k$ satisfies the integral equation
\begin{align} \label{vk-rad}v_k(x):=\frac{t_k}{\gamma_n}\int_{\R^n}\log\left(\frac{1}{|x-y|}\right)|y|^{n\alpha}e^{-n|y|^2}e^{n(v_k(y)+c_{v_k})}dy.\end{align}
We observe that if
$$w_k:=v_k+c_{v_k}+\frac1n\log t_k\leq C\quad \text{on }\R^n,$$
then from \eqref{vk-rad}
$$|v_k(x)|\leq \frac{e^{nC}}{\gamma_n}\int_{\R^n}|\log|x-y|||y|^{n\alpha}e^{-n|y|^2}dy\leq C\log(2+|x|),$$
a contradiction to our assumption $\|v_k\|\to\infty$. Thus, together with Lemma \ref{lemmamono}, we have 
$$\max_{\R^n} w_k=w_k(0)\to\infty.$$ We set $$\eta_k(x)=w_k(r_kx)-w_k(0),\quad r_k^{1+\alpha}:=e^{-w_k(0)}\to0.$$
Then, {on  compact sets} we have
$$(-\Delta)^\frac n2 \eta_k(x)=|x|^{n\alpha}e^{-n|r_k x|^2} e^{n\eta_k(x)}=(1+o(1))|x|^{n\alpha}e^{n\eta_k(x)}.$$
Moreover by Lemma \ref{diffint} we have
\begin{equation}\label{boundsetak}
\int_{B_R}|\De \eta_k(x)|dx\leq CR^{n-2},\quad \int_{\R^n\setminus B_R}\frac{|\De \eta_k(x)|}{|x|^{2n-2}}dx\leq \frac{C}{R^n}.
\end{equation}
Then by elliptic estimates, Proposition \ref{stimaell} in the appendix, up to a subsequence, $\eta_k\to\eta$ locally uniformly $ \R^n$ and in $C^{n-1}_{\loc}(\R^n\setminus \{0\})$.  
Note further that we have 
\begin{align}\label{L1 bound}
\int_{\R^n}|x|^{n\alpha}e^{n\eta(x)}dx&=\lim_{R\to\infty}\lim_{k\to\infty}\int_{B_R}|x|^{n\alpha}e^{n\eta_k(x)}dx\nonumber\\ &=\lim_{R\to\infty}\lim_{k\to\infty}t_k\int_{B_{Rr_k}}|x|^{n\alpha}e^{-n|x|^2}e^{n(v_k(x)+c_{v_k})}dx \nonumber\\&\leq \Lambda t_\infty.
\end{align}
where, up to a subsequence, $t_\infty:=\lim_{k\to\infty}t_k$. Thus $t_\infty\neq 0$ and $$\int_{\R^n}|x|^{n\alpha}e^{n\eta}dx<\infty.$$  We claim that  $\eta$ is a normal solution to $(-\De)^\frac n2\eta=|x|^{n\alpha}e^{n\eta}$, namely
\begin{equation}\label{limeta}
\eta(x)=\frac{1}{\gamma_n}\int_{\R^n}\log\left(\frac{1+|y|}{|x-y|}\right)|y|^{n\alpha}e^{n\eta(y)}dy+c,
\end{equation}
{ for some $c\in \R$.}
Since $n\ge 3 $ we have that $\Delta \eta_k\to \Delta \eta$ locally uniformly in $\R^n\setminus \{0\}$ and by the first estimate in \eqref{boundsetak} we conclude that $\Delta\eta_k\to \Delta \eta$ in $L^1_{\loc}(\R^n)$. Then, also using the second estimate in \eqref{boundsetak} and recalling that for $\varphi\in \mathcal{S}(\R^n)$ we have
$$|(-\De)^\frac{n-2}{2}\vp |\le C (1+|x|)^{-2n+2},$$
we get
 \begin{align}\label{eqlimeta}
 \int_{\R^n}(-\De\eta)(-\De)^\frac{n-2}{2}\vp dx&=\lim_{k\to\infty}\int_{\R^n}(-\De\eta_k)(-\De)^\frac{n-2}{2}\vp dx \nonumber \\ &=\lim_{k\to\infty}\int_{\R^n}\eta_k(-\De)^\frac n2\vp dx\nonumber \\ &=\lim_{k\to\infty}\int_{\R^n}|x|^{n\alpha}e^{-n|r_kx|^2}e^{n\eta_k}\vp dx\nonumber \\&=\int_{\R^n}|x|^{n\alpha}e^{n\eta}\vp dx.
\end{align}
In order to prove that \eqref{eqlimeta} implies \eqref{limeta} (compare to Definitions \ref{def-soln} and \ref{normal}), we set
$$\tilde \eta(x):=\frac{1}{\gamma_n}\int_{\R^n}\log\left(\frac{1+|y|}{|x-y|}\right)|y|^{n\alpha}e^{n\eta(y)}dy,\quad p:=\eta-\tilde \eta.$$
Then $p$ satisfies
$$\int_{B_R}|\De p|dx\leq   CR^{n-2},\quad\int_{\R^n}\frac{|\De p|}{1+|x|^{2n-2}}dx<\infty,\quad \int_{\R^n}\De p(-\De)^\frac{n-2}{2}\vp dx=0,$$
for every $\vp\in \mathcal{S}(\R^n)$. Hence $\De p\in L_\frac{n-2}{2}(\R^n)$ is $\frac{n-2}{2}$-harmonic in $\R^n$, which implies that $\De p$ is a polynomial  (see e.g. the proof of \cite[Lemma 2.4]{H-clas}). Now  the estimate $\int_{B_R}|\De p|dx\leq CR^{n-2}$ gives $\De p\equiv 0$. Since  $\tilde\eta(x)\geq -C\log(2+|x|)$ on $\R^n$ (see the proof of Lemma \ref{below}), and $|\cdot|^{n\alpha}e^{n\eta}\in L^1(\R^n)$, by Theorem \ref{thm6} we get $p\equiv const$, and \eqref{limeta} follows.

From {the Pohozaev identity} of Proposition  \ref{special-poho} (case $\mu=0$) and \eqref{L1 bound}, we infer
$$\Lambda_1(1+\alpha)=\int_{\R^n}|x|^{n\alpha}e^{n\eta}dx
\le \Lambda,$$
which contradicts our assumption $\Lambda<\Lambda_1(1+\alpha)$. 
\end{proof}

\noindent\textbf{Proof of Theorem \ref{thm4} (completed).}
Assume first $\Lambda\in (0,\Lambda_1(1+\alpha))$. Thanks to Proposition \ref{aprioriTv} we can apply Schauder's theorem, hence $T$ has a fixed point $v$. Then the function $u(x)=v(x)-|x|^2+c_v$ satisfies \eqref{P0}, as wished.

\medskip

Now we consider the case $\Lambda=\Lambda_1(1+\alpha)$. We fix a sequence $\Lambda_k\uparrow \Lambda$, and for each $k$ we apply the previous procedure to find $v_k$ fixed point of the corresponding $T_{\Lambda_k}$. We claim that
\begin{equation}\label{vkinfty}
v_k(0)+c_{v_k}\to\infty.
\end{equation}
Otherwise, from \eqref{vk-rad} we would infer that $v_k$ satisfies $|v_k(x)|\leq C\log(2+|x|)$ on $\R^n$. { Then from the definition of  $c_{v_k}$  we  get $|c_{v_k}|\leq C$. Moreover $(v_k)$ is equicontinuous on $\R^n$, and therefore,  up to a subsequence, $v_k\to v$ locally uniformly in $\R^n$.   
The} limit function $v$ satisfies $$v(x) =\frac{1}{\gamma_n}\int_{\R^n}\log\left(\frac{1}{|x-y|}\right)|y|^{n\alpha}e^{-n|y|^2}e^{n(v(y)+c)}dy,$$ where $c_{v_k}\to c$ and
$$\int_{\R^n}|y|^{n\alpha}e^{-n|y|^2}e^{n(v(y)+c)}dy{\le }\lim_{k\to\infty }\Lambda_k=\Lambda_1(1+\alpha),$$
a contradiction to Proposition \ref{special-poho} (case $\mu=n$). This proves \eqref{vkinfty}.

Setting
$$\eta_k(x)=v_k(r_kx)-v_k(0),\quad r_k^{1+\alpha}:=e^{-v_k(0)-c_{v_k}}\to0,$$
as in the proof of Lemma \ref{aprioriTv}, one obtains $\eta_k\to\eta$, where $\eta$ is a normal solution to \eqref{P0}. \hfill $\square$







\section{Proof of Theorem \ref{thm2}} 

We set $$X:=\left\{v\in C^0(\R^n):\|v\|<\infty \right\},\quad \|v\|:=\sup_{x\in\R^n}\frac{|v(x)|}{1+|x|}.$$ We fix $\Lambda\in (0,\Lambda_1)$ and $\alpha\in (-1,\infty)$.  For $v\in X$ let $c_v\in\R$ be determined by \begin{align}\label{Lambda}\int_{\R^n}|y|^{n\alpha}e^{-n|y|^2}e^{n(v(y)+c_v)}dy=\Lambda,\end{align} and let $v^*\in\R$ be given by $$v^*:=\sup_{|x|\leq 1}e^{\frac{2}{1+\alpha}(v(x)+c_v)}.$$ We define $T:X\to X$, $v\mapsto \bar v$ where $$\bar v(x):=\frac{1}{\gamma_n}\int_{\R^n}\log\left(\frac{1+|y|}{|x-y|}\right)|y|^{n\alpha}e^{-n|y|^2}e^{n(v(y)+c_v)}dy+v^*x_1.$$ 
Notice that
$$(-\Delta)^\frac n2 \bar v(x)=|x|^{n\alpha}e^{n(v(x)-|x|^2+c_v)}.$$
We will look for solutions of the form
$$u(x)=v(x)-|x|^2+c_v.$$


\begin{lem}\label{lemmacomp}
The operator $T:X\to X$ is compact. 
\end{lem}

\begin{proof} As in Lemma \ref{lemmacomp0} continuity follows by dominated convergence. Moreover given a bounded sequence $(v_k)\subset X$, from the definition of $c_{v_k}$ it follows that $|c_{v_k}|\le C$. This in turn implies $|v_k^*|\le C$. If we set $\tilde v_k(x):=\bar v_k(x)- v_k^*x_1,$ we get $|\tilde v_k(x)|\leq C\log (2+|x|)$, and the sequence $(\tilde v_k)$ is equicontinuous, with the same proof as in Lemma \ref{lemmacomp0}.
In particular, up to a subsequence, $\tilde v_k\to \tilde v$ in $X$.  Since, up to a subsequence, $\|v_k^* x_1 - c_0 x_1\|\to 0$ for some $c_0> 0$, we conclude that $\bar v_k\to \tilde v+c_0x_1$ in $X$.   
\end{proof}

The proof of Theorem \ref{thm2} follows at once from the Schauder fixed-point theorem and the following a priori estimate.

\begin{prop}\label{propboundv}
There exists $C>0$ such that $$\|v\|\leq C\quad\text{for every }(v,t)\in X\times (0,1]\text{ with } v=tT(v).$$
\end{prop}

The proof of Proposition \ref{propboundv} will be based { on} the following three lemmas.

\begin{lem}\label{local}
For every $R>0$ there exists a constant $C(R)>0$ such that for every $(t,v)\in (0,1]\times X$ {with}  $v=t T(v)$, that is
\begin{equation}\label{eqvk}
v(x):=\frac{t}{\gamma_n}\int_{\R^n}\log\left(\frac{1+|y|}{|x-y|}\right)|y|^{n\alpha}e^{-n|y|^2}e^{n(v(y)+c_{v})}dy+tv^*x_1,
\end{equation}  
we have
$$w:=v+c_{v}+\frac1n \log t \le C(R)\quad \text{on }B_R.$$
\end{lem}
\begin{proof} 
Assume by contradiction that for a sequence $(t_k,v_k)\in (0,1]\times X$ such that $v_k= t_k T(v_k)$ one has 
$$\max_{\bar B_{2R_0}}w_k=:w_k(\xi_k)= v_k(\xi_k)+c_{v_k}+\frac{1}{n}\log t_k \to\infty$$ for some $R_0>0$.  We set $x_k:=\xi_k$ and $s_k=0$  if $|\xi_k|\not\to 2R_0$,  and otherwise we let $x_k\in B_{R_0}(\xi_k)$ and $s_k\in[0,R_0]$ be such that  (see  \cite{ARS}) $$(R_0-s_k)e^{w_k(x_k)}=(R_0-s_k)\max_{\bar B_{s_k}(\xi_k)}e^{w_k}=\max_{s\in [0,R_0]}\left((R_0-s)\max_{\bar B_{s}(\xi_k)}e^{w_k}\right) =:L_k.$$
Then $L_k\to\infty$, and \begin{align}\label{muk}w_k(x_k+\mu_kx)-w_k(x_k)\leq \log2\quad\text{for }|x|\leq \frac{L_k}{2},\quad \mu_k:=\frac{R_0-s_k}{L_k}.\end{align}

We distinguish the following cases. 

\medskip 
\noindent\textbf{Case 1} Up to a subsequence $|x_k|^{1+\alpha}e^{w_k(x_k)}\to c_0\in [0,\infty)$. 

In this case we have $x_k=\xi_k\to 0$, which implies that $w_k\leq w_k(x_k)$ on $B_{R_0}$. We set
$$\eta_k(x):=w_k(r_k x)-w_k(x_k),\quad r_k^{1+\alpha}:=e^{-w_k(x_k)}.$$
It follows from the definition of $r_k$ that $|x_k|=O(r_k)$. Therefore,  on any compact set
\begin{equation}\label{eqetak}
(-\De)^\frac n2\eta_k(x)=\left|x\right|^{n\alpha} e^{n\eta_k(x)}e^{-n|r_kx|^2} {= (1+o(1))\left|x\right|^{n\alpha} e^{n\eta_k(x)} }.
\end{equation}
{ Since $\alpha>-1$ and  $\eta_k\le \log 2$ for large $k$ on any compact set}, from \eqref{eqetak} we obtain
$$\|(-\De)^\frac n2\eta_k\|_{L^p(B_R)}\le C(p,R)\quad \text{for }1\le p<\bar p,$$
and
$$\|(-\De)^\frac n2\eta_k\|_{L^\infty(K)}\le C(K)\quad \text{for every }K\Subset\R^n\setminus\{0\}.$$
Moreover, differentiating in \eqref{eqvk} as in Lemma \ref{diffint} (notice that the part $v^*x_1$ does not play a role in $\De\eta_k$), we obtain
\begin{equation}\label{stimeinteg}
\int_{B_R}|\De \eta_k(x)|dx\leq CR^{n-2},\quad  \int_{\R^n\setminus B_R}\frac{|\De \eta_k(x)|}{|x|^{2n-2}}dx\leq \frac{C}{R^n}.
\end{equation}
We also have $\eta_k(\bar x_k)=0$, where $\bar x_k:=\frac{x_k}{r_k}$ satisfies $|\bar x_k|=O(1)$. Then by Proposition \ref{stimaell}, up to a subsequence, $\eta_k\to\eta$ in $ C^0_{\loc}(\R^n)\cap C^{n-1}_{\loc}(\R^n\setminus\{0\})$ and $\Delta \eta_k \to \Delta \eta$ in $L^1_{\loc}(\R^n)$ for some $\eta$.  

Then, with the same argument of Lemma \ref{aprioriTv} we obtain that $\eta$ satisfies the integral equation
\begin{equation}\label{limeta bis}
\eta(x)=\frac{1}{\gamma_n}\int_{\R^n}\log\left(\frac{1+|y|}{|x-y|}\right)|y|^{n\alpha}e^{n\eta(y)}dy+c.
\end{equation}
In particular, differentiating \eqref{limeta bis} we obtain that for every $R>0$
$$\int_{B_R}|\nabla \eta(x)|dx\leq CR^{n-1}.$$ 
Using \eqref{eqvk} one obtains for every $R>0$
\[\begin{split}
\int_{B_R}|\nabla \eta_k (x)-t_kr_kv_k^*e_1|dx&\leq t_k\int_{B_R}\int_{\R^n}\frac{r_k}{|x_k+r_kx-y|} |y|^{n\alpha}e^{-n|y|^2}e^{n(v_k(y)+c_{v_k})}dydx\\
&\leq CR^{n-1}.
\end{split}\]
Therefore, for every $R>0$
\begin{align*}
\lim_{k\to\infty}t_kr_kv_k^*|B_R|\leq \lim_{k\to\infty}\int_{B_R}\left(|\nabla \eta_k (x)-t_kr_kv_k^*e_1| +|\nabla\eta_k(x)|\right)dx\leq CR^{n-1}.
\end{align*}
This shows that $$\lim_{k\to\infty}t_kr_kv_k^*=0.$$  Since $t_k\to t_\infty\neq 0$, we must have 
$$\lim_{k\to\infty}r_kv_k^*=0.$$  This is a contradiction since from $\liminf_{k\to \infty}t_k>0$ we infer $v_k^*\ge \frac{1}{C} e^{\frac{2}{1+\alpha}w_k(x_k)}$, hence
$$ r_kv_k^* \ge \frac{1}{C}e^{\frac{1}{1+\alpha}w_k(x_k)}\to\infty. $$

\medskip 

\noindent\textbf{Case 2} Up to a subsequence $|x_k|^{1+\alpha}e^{w_k(x_k)}\to \infty$.

We set
 $$\eta_k(x)=w_k(x_k+r_kx)-w_k(x_k),$$  where
 $$r_k:=|x_k|^{-\alpha}e^{-w_k(x_k)}.$$ 
Notice that by \eqref{muk} for every $R>0$ we have $\eta_k(x) \le \log 2$ on $B_R$ for $k\ge k_0(R)$. Moreover $r_k=o(|x_k|)$ and we compute
$$(-\De)^\frac n2\eta_k(x)=e^{-n|x_0|^2}(1+o(1))\left|\frac{x_k}{|x_k|}+\frac{r_k}{|x_k|}x\right|^{n\alpha}e^{n\eta_k}=(c_0+o(1)) e^{n\eta_k},$$
{on compact sets, where $x_k\to x_0$ and}  $c_0:= e^{-n|x_0|^2}$.
Then, similar to Case $1$, we obtain $\eta_k\to\eta$ where $\eta$ is {a} normal solution to
$$(-\De)^\frac n2\eta=c_0e^{n\eta}\quad\text{in }\R^n,\quad \int_{\R^n}e^{n\eta}dx<\infty,$$
that is, $\eta$ is a  spherical, a contradiction as $\Lambda<\Lambda_1$. 
\end{proof}

\begin{lem}\label{infty}
There exists $C>0$ such that for every $(t,v)\in (0,1]\times X$ such that $v=tT(v)$ one has $v^*\leq C$ and $c_{v}\leq C$.
 \end{lem}
 \begin{proof}
Take $(t,v)\in (0,1]\times X$ be such that $v=tT(v)$ and let $w$ be as in Lemma \ref{local}. For $|x|\leq 1$ we obtain from \eqref{eqvk} 
 \begin{align}
 v(x)-tv^*x_1&=\frac{1}{\gamma_n}\left(\int_{|y|< 2}+\int_{|y|>2}\right)\log\left(\frac{1+|y|}{|x-y|}\right)|y|^{n\alpha}e^{-n|y|^2}e^{n w(y)}dy\notag\\
 &=O(1)\int_{|y|<2}\log\left(\frac{1+|y|}{|x-y|}\right)|y|^{n\alpha}dy+O(1)\int_{|y|>2}|y|^{n\alpha}e^{-n|y|^2}e^{n w(y)}dy\notag\\
 &=O(1),\label{vk-local}
 \end{align}
{ where  the  last equality  follows from  \eqref{Lambda}, while in the second inequality we have used that $w\leq C$ on $B_2$ thanks to Lemma \ref{local}, and the estimate $\log\left(\frac{1+|y|}{|x-y|}\right)=O(1)$ for $(x,y)\in B_1\times B_2^c$.} Therefore, $$v(x)+c_{v}=c_{v}+tv^*x_1+O(1)\quad \text{on }B_1.$$  This  and  \eqref{Lambda} imply that  $v+c_{v}\leq C$ on $B_1$, which is equivalent to  $v^*\leq C$. In particular, from  \eqref{vk-local}, we have $v=O(1)$ {in $B_1$}, and  \eqref{Lambda}  yields ${c_{v}}\leq C$.
\end{proof}  

\begin{lem}\label{inftybis}
For every $\ve>0$ there exists $R>0$ such that the following holds: Given $(t,v)\in (0,1]\times X$ with $v=tT(v)$ one has
$$\int_{|x|>R}e^{-n|x|^2}|x|^{n\alpha}e^{nw(x)}dx<\ve,$$
{where $w:=v+c_{v}+\frac 1n\log t$.}
\end{lem}
\begin{proof}
We use a Pohozaev type identity  for the integral equation 
 $$\tilde w(x)=\frac{1}{\gamma_n}\int_{\R^n}\log\left(\frac{1}{|x-y|}\right)e^{-n|y|^2-tv^*y_1}|y|^{n\alpha}e^{n\tilde w(y)}dy+nc_{v}+\frac1n\log t,$$  where $\tilde w:=w+tv^*x_1$. Since $(t,v)\in (0,1]\times X$, we can apply Proposition \ref{poho-3} to get 
\begin{align*} t\Lambda(t\Lambda-\Lambda_1)&=c\int_{\R^n}y\cdot\nabla\left( e^{-n|y|^2-tv^*y_1}|y|^{n\alpha} \right)e^{n\tilde w(y)}dy\\
&=cn\alpha\int_{\R^n}|y|^{n\alpha}e^{-n|y|^2}e^{nw(y)}dy-c\int_{\R^n}(2n|y|^2+tv^*y_1) e^{-n|y|^2}|y|^{n\alpha}e^{nw(y)}dy\\
&=cn\alpha t\Lambda-c\left(\int_{B_{R_0}}+\int_{B_{R_0}^c}\right)(2n|y|^2+tv^*y_1)e^{-n|y|^2}|y|^{n\alpha}e^{nw(y)}dy\\
&=:cn\alpha t\Lambda-c(I_1+I_2),
\end{align*}
where $R_0>0$ is such that $2n|y|^2+tv^*y_1\geq |y|^2$ on $B_{R_0}^c$. We observe that $|I_1|\leq C(R_0,\Lambda,n)$, thanks to Lemma \ref{local} and the estimate $v^*\leq C$ {of Lemma \ref{infty}}. Therefore, for $R>R_0$  we obtain \begin{align*}\int_{B_R^c} e^{-n|y|^2}|y|^{n\alpha}e^{nw(y)}dy&\leq \frac{1}{R^2} \int_{B_R^c} |y|^2e^{-n|y|^2}|y|^{n\alpha}e^{nw(y)}dy\\ &\leq \frac{1}{R^2}  \int_{B_R^c} (2n|y|^2+tv^*y_1)e^{-n|y|^2}|y|^{n\alpha}e^{nw(y)}dy\\ &\leq \frac{1}{R^2}  \int_{B_{R_0}^c} (2n|y|^2+tv^*y_1)e^{-n|y|^2}|y|^{n\alpha}e^{nw(y)}dy\\ & \leq \frac{1}{R^2}C(R_0,\Lambda,n).\end{align*}
This proves the lemma. 
 \end{proof}

\noindent\emph{Proof of Proposition \ref{propboundv}.}
Since $v^*\leq C$, thanks to Lemma \ref{infty},  it is sufficient to show that $\tilde v:=v-tv^*x_1$ is bounded in $X$ (we want to show that $|\tilde v(x)|\leq C\log (2+|x|)$). We have 
\begin{equation}\label{tildevk}
\tilde v(x):=\int_{\R^n}\log\left(\frac{1+|y|}{|x-y|}\right)K(y)e^{n\tilde v(y)}dy,\quad K(y):=\frac{t_k}{\gamma_n}|y|^{n\alpha}e^{-n|y|^2}e^{nc_{v}+ntv^* y_1}. \end{equation}  As {$v^*\le C$ and} $c_{v}\leq C$, there exists $R>0$ such that
\begin{align}\label{Kk}K(y)\leq e^{-2|y|^2}\quad\text{on }B_R^c.\end{align}
By Lemma \ref{inftybis} we can also assume that $$\int_{B_R^c}Ke^{n\tilde v}dy\leq  \frac14. $$ Then,  as in  Lemma \ref{upv} one { obtains} \begin{align}\label{vk-upper} \tilde v(x)\leq (-t\Lambda+\frac 14)\log |x|+\int_{B_1(x)}\log\left(\frac{1}{|x-y|}\right)K(y)e^{n\tilde v(y)}dy,\quad |x|\geq R.\end{align} In the spirit of Lemma \ref{lpv} we  {get } $$\int_{B_1(x_0)}e^{2n\tilde v(z)}dz\leq C|x_0|^\frac14,\quad |x_0|\geq R+2.$$   By \eqref{Kk} and  H\"older inequality, from \eqref{vk-upper} $$\tilde v(x)\leq \frac14\log |x| +C,\quad |x|\geq R+2.$$ Therefore, $$|\tilde v(x)|\leq C\int_{\R^n}\left|\log\left(\frac{1+|y|}{|x-y|}\right)\right| e^{-|y|^2}dy\leq C\log (2+|x|),$$ thanks to Lemma \ref{local}.
\hfill$\square$

\appendix

\section{Appendix}
\subsection{A Pohozaev-type result}

\begin{prop}\label{poho-3}
{Assume that  $K\in W^{1,1}_{\loc}(\R^n \setminus \{0\})\cap L^p_{\loc}(\R^n)$ for some $p>1$, $n\ge 2$.} Let $\eta$ be a solution to the integral equation 
\begin{equation}\label{int eq}
\eta(x)= \frac{1}{\gamma_n}\int_{\R^n}  \log\left(\frac{1+|y|}{|x-y|}\right)K(y) e^{n \eta(y)} dy + c 
\end{equation}
for some $c\in \R$, {with $Ke^{n \eta}\in L^1(\R^n)$ and $(\nabla K \cdot x) \, e^{n \eta} \in L^1(\R^n)$}. If there exists $R_0,\eps >0$ such that 
 \begin{align}\label{eta-est bis}
 | K (x)|e^{n\eta(x)} \le \frac{1}{|x|^{n+\eps}} \quad\text{for } |x|\geq R_0,
  \end{align} 
  then, {denoting $\Lambda:=\int_{\R^n} K(x)e^{n\eta(x)} dx,$} we have
\begin{equation}\label{general Poho}
\frac{\Lambda}{\gamma_n}\left(\Lambda-2\gamma_n\right) =\frac{2}{n }\int_{\R^n}\left(x\cdotp\nabla K(x )\right)e^{n\eta(x)}dx,
  \end{equation}
\end{prop}
\begin{proof}
 Noticing that ${\eta\in C^0(\R^n)\cap W^{1,1}_{\loc}(\R^n)}$, in the spirit of  \cite[Theorem 1.1]{Xu}, for any $R>0$ we can multiply $\nabla \eta $ by $x\cdot \nabla \eta $ and integrate on $B_R$ using the divergence theorem to  get  
 \begin{align*}&\int_{B_R}Ke^{n\eta}dx+\frac1n\int_{B_R}\left(x\cdot\nabla K(x)\right)e^{n\eta}dx-\frac Rn\int_{\partial B_R}Ke^{n\eta}d\sigma(x) \\ &= \frac{1}{2\gamma_n}\int_{ B_R}\int_{B_R^c}\frac{(x-y)\cdot(x+y)}{|x-y|^2}K(y)e^{n\eta(y)}K(x)e^{n\eta(x)}dydx \\ &\quad +\frac{1}{2\gamma_n}\int_{\R^n}K(y)e^{n\eta(y)}dy\int_{B_R}K(x)e^{n\eta(x)}dx.\end{align*} 
It follows from \eqref{eta-est bis} that the boundary term and the double integral term  go to $0$ as $R\to\infty$. Therefore,  taking $R\to\infty$, we obtain  \eqref{general Poho}.

\end{proof}

We are interested in the following special cases {of the above proposition. }

\begin{prop}\label{special-poho} 
Given $n\ge 2$, $\alpha>-1$, $\mu\ge 0$  let $\eta$ be a solution to \eqref{int eq} with $c\in \R$,  $K := |\cdot|^{n\alpha}e^{-\mu |\cdot|^2}$ and 
$$
\Lambda:=\int_{\R^n} K(x)e^{n\eta(x)} dx <+\infty.
$$
Then, $\Lambda  \le  \Lambda_1(1+\alpha)$ and the equality holds if and only if $\mu =0$.  
\end{prop}
\begin{proof}
First, we claim that \eqref{eta-est bis} holds. If $\mu=0$, then $\eta$ is a normal solution to $(-\Delta)^\frac{n}{2}\eta = |x|^{n\alpha}e^{n \eta}$ and the classification part of Theorem \ref{mainthm} implies that 
\begin{equation}\label{Eta Log}
\lim_{|x|\to\infty}\frac{\eta(x)}{\log|x|}=-\beta, \quad \beta:=\frac{\Lambda}{\gamma_n}.
\end{equation} 
Moreover, arguing as in Lemma \ref{below}   we get
\begin{equation}
v(x)\ge -\beta \log|x|+c,  \quad |x|\ge 1,
\end{equation}
and from  $K e^{n\eta}\in L^1$ we find that $\beta > 1+\alpha$. Then \eqref{eta-est bis} follows at once form \eqref{Eta Log}.  If $\mu>0$,  we get \eqref{Eta Log} from Remark \ref{newrmk} and the function $K e^{n\eta}$ decays exponentially at infinity, so that  \eqref{eta-est bis} trivially holds.  
Observe now that {the condition $(\nabla K \cdot x ) e^{n\eta}\in L^1(\R^n)$ is satisfied}, since $\nabla K \cdot x = n\alpha K$ for $\mu=0$, and since $(\nabla K \cdot x ) e^{n\eta}$ decays exponentially for $\mu >0$ thanks to \eqref{Eta Log}.
Then, we can apply Proposition \ref{poho-3} to get 
\begin{align*}  \frac{\Lambda}{\gamma_n}\left(\Lambda-2\gamma_n\right)&
            =\frac{2}{n }\int_{\R^n}\left(x\cdotp\nabla\left(|x|^{n\alpha}e^{-\mu|x|^2}\right)\right)e^{n\eta(x)}dx\\
       &   =2\alpha\Lambda-\frac{4\mu}{n }\int_{\R^n}|x|^{n\alpha+2}e^{-\mu|x|^2}e^{n\eta(x)}dx\\
       & \le 2\alpha\Lambda,
       \end{align*}
       with equality holding iff $\mu =0$. Since $\Lambda_1=2\gamma_n$, the proof is complete. 
\end{proof}


\subsection{A Liouville-type theorem}

\begin{lem}\label{u+ to 0}
Let $\alpha\in(-1,0)$ and $u$ be {a measurable function such that $|\cdot|^{n\alpha} e^{n u}\in L^1(\R^n)$}. Then for any $x\in\rn$ we have
$$\fint_{B_R (x)} u^+\, dy\to 0$$
as $R\to\infty$.
\end{lem}
\begin{proof}
Fix $x\in\rn$. With $nu^+\leq e^{n u}$, multiplying and dividing by $|y|^{n\alpha}$ we get
\[
\begin{split}
n\fint_{B_R(x)}u^+\, dy&\leq \fint_{B_R(x)}e^{nu}\, dy\\&
\leq \frac{C (R+|x|)^{-n\alpha}}{R^{n}}\int_{B_R(x)}|y|^{n\alpha}e^{nu}\, dy\\&
\leq  \frac{C (R+|x|)^{-n\alpha}}{R^{n}},
\end{split}
\]
where we used that for $y\in B_R(x)$ we have $|y|\leq R+|x|$ and that $\int_{\rn} |x|^{n\alpha}e^{nu}\, dx<\infty$. The claim follows letting $R\to\infty$ since $\alpha\in(-1,0)$. 
\end{proof}

{
\begin{thm}\label{thm6}
Let $\alpha>-1$, $m\ge 1$ and consider $h\colon \rn\to \R$ with $(-\Delta) ^\frac{m}{2} h=0$ in the sense of Definition \ref{def-soln}. If $m$ is even, assume further that $h(x)\leq u+C\log(1+|x|)+C$ for any $x\in \R^n$, with $\int_{\rn} |x|^{n\alpha} e^{nu}\, dx<\infty$. Then $h$ is a polynomial of degree at most $m-1$.
\end{thm}
\begin{proof}
If $m\ge 1$ is even, the proof is almost identical to the one of Theorem $6$ in \cite{MarClass}, the only difference being the estimate of the term $\fint_{B_R(x)} u^+\, dy\to 0$ for $\alpha\in (-1,0)$, which is true thanks to Lemma \ref{u+ to 0}. 
In the case $m\geq 1$ odd, notice that $h\in  L_{\frac m2}(\R^n)$ by Definition \ref{def-soln}.  This implies that  $h$ is a polynomial of degree at most $m-1$ (see e.g. the proof of \cite[Lemma 2.4]{H-clas}).
\end{proof}
}

\subsection{A non-local elliptic estimate}

%

\begin{prop}\label{stimaell}
Assume $n\ge 3$. Let $(u_k)\subset L_{\frac{n}{2}}(\R^n)$ be a sequence of solutions to
$$(-\De )^\frac n2u_k=f_k\quad \text{in }\R^n$$
for some $f_k\in L^1(\R^n)$ satisfying \begin{align}\|f_k\|_{L^1(\R^n)}\leq C,\quad \|f_k\|_{L^p(B_R)}\leq C,\quad \|f_k\|_{L^\infty(A)}\leq C, \end{align} for some $p>1$, $R>0$ and an open set $A\Subset \R^n\setminus\{0\}$. Assume further that
$$u_k\leq u_k(0)=0\text{ in } B_R\quad\text{and } \int_{B_R}|\De u_k|dx\leq C.$$
Then the sequence $(u_k)$ is bounded $C^{0,\sigma_1}_{\loc}(B_R)$ and in $C^{n-1,\sigma_2}_{\loc}(A)$ for some $\sigma_1=\sigma_1(p)\in (0,1)$ and for every $\sigma_2\in (0,1)$. 
\end{prop} 
\begin{proof} 
We set \begin{align}\label{vk} v_k(x):=\frac{1}{\gamma_n}\int_{\R^n}\log\left(\frac{1+|y|}{|x-y|}\right)f_k(y)dy,\quad p_k:=u_k-v_k.\end{align}  Then by \cite[Lemma 2.4]{H-clas} we have that $p_k$ is a polynomial of degree at most $n-1$. It follows from the assumptions on $f_k$ and from \eqref{vk} that $(v_k)$ is bounded in $C^{0,\sigma_1}_{\loc}(B_R)$ and in $C^{n-1,\sigma_2}_{\loc}(A)$.
Therefore, $p_k$ satisfies $$\sup_{B_R}p_k\leq C(R),\quad |p_k(0)|\leq C,\quad \int_{B_\frac R2}|\De p_k|dx\leq C.$$ Hence, $(p_k)$ is bounded in $C^\ell_{\loc}(\R^n)$ for every $\ell\in\mathbb{N}$.
\end{proof}

\subsection{Some integral estimates}

\begin{lem}\label{log}
There exists $C>0$ such that for every $\rho\in (0,1]$ we have for any $x, \xi\in\R^{n}$
\begin{itemize}
\item[i)] $$\int_{B_\rho(x)}\int_{B_\rho(x)}\left|\log\bra{\frac{|z-\xi|}{|y-\xi|}}\right|dzdy\leq C\rho^{2n}.$$
\item[ii)]  If $|x-\xi|\geq { 2\sqrt{\rho}}$ then  $$\int_{B_\rho(x)}\int_{B_\rho(x)}\left|\log\bra{\frac{|z-\xi|}{|y-\xi|}}\right|dzdy\leq C\rho^{2n+\frac12}.$$
\item[iii)] $$\int_{B_\rho(x)}\int_{B_\rho(x)}\left|\log\bra{\frac{|z-\xi|}{|y-\xi|}}\right|dzdy=o(1)\rho^{2n},\quad o(1)\xrightarrow{|x-\xi|\to\infty}0.$$
\end{itemize}
\end{lem}
\begin{proof}
{Under the rescaling}
$${y'=\frac{y-x}\rho\quad\quad\quad z'=\frac{z-x}\rho\quad\quad\quad\xi':=\frac{\xi-x}\rho,}$$
{$i)$ will be equivalent to showing}
$${\int_{B_1(0)}\int_{B_1(0)}\left| \log\left(\frac{|z'-\xi'|}{|y'-\xi' |}\right)\right|dz' dy'\le C,}$$
{independently on $x,\xi'$.\\
If $|\xi'|\ge2$, then we suffice to apply the triangular inequality to get $\left| \log\left(\frac{|z'-\xi'|}{|y'-\xi' |}\right)\right|\le\log3$.\\
On the other hand, for $|\xi'|\le2$ one has}
\[
\begin{split}
{\int_{B_1(0)}\int_{B_1(0)}\left| \log\left(\frac{|z'-\xi'|}{|y'-\xi' |}\right)\right|dz' dy'}
&{\le\int_{B_3(\xi')}\int_{B_3(\xi')}\left| \log\left(\frac{|z'-\xi'|}{|y'-\xi' |}\right)\right|dz' dy'}\\
&=\int_{B_3(0)}\int_{B_3(0)}\left| \log\left(\frac{|z''|}{|y'' |}\right)\right|dz'' dy''\\
&{\leq 2|B_3(0)|\int_{B_3(0)}|\log|z''||dz'''}\\
&\le C.
\end{split}
\]
To prove $ii)$ first notice that for $|x-\xi|\geq { 2\sqrt\rho}$ one has 
$$\frac{|z-\xi|}{|y-\xi|}\leq\frac{|y-\xi|+|z-y|}{|y-\xi|}\leq 1+ 2\sqrt\rho.$$
 Exchanging the role of $y$ and $z$ we find the same estimate for $\frac{|y-\xi|}{|z-\xi|}$.  
Therefore 
$$
\left|\log\bra{\frac{|z-\xi|}{|y-\xi|}}\right|\leq  \log(1+ 2\sqrt\rho) \leq 2 \sqrt\rho.
$$
The proof follows immediately. 

As for the proof of $iii)$ it suffices to notice that as $|x-\xi|\to\infty$ we have
$$\frac{|z-\xi|}{|y-\xi|}\to 1 \quad \text{uniformly for } y,z\in B_\rho(x)\subset B_1(x).$$
\end{proof}


\begin{thebibliography}{10}
\small

\bibitem{ARS}
 \textsc{Adimurthi, F. Robert, M. Struwe}, \emph{Concentration phenomena for Liouville's equation in dimension $4$}, J. Eur. Math. Soc. \textbf{8}, (2006), 171-180. 

\bibitem{Bar}\textsc{D. Bartolucci}, \emph{On the best pinching constant of conformal metrics on $S^2$ with one and two conical singularities}, J. Geom. Anal. \textbf{23} (2013), no. 2, 855-877.

\bibitem{Br} \textsc{T. Branson}, \emph{Sharp inequality, the functional determinant and the complementary series}, Trans. Amer. Math. Soc. \textbf{347} (1995), 3671-3742.

\bibitem{BO}\textsc{T. Branson, B. Oersted}, \emph{Explicit functional determinants in four dimensions}, Comm. Partial Differential Equations \textbf{16} (1991), 1223-1253.
\bibitem{CC} \textsc{S.-Y. A. Chang, W. Chen}, \emph{A note on a class of higher order conformally covariant equations}, Discrete Contin. Dynam. Systems \textbf{7} (2001), no. 2, 275-281.


\bibitem{CL} \textsc{W. Chen, C. Li}, \emph{Qualitative properties of solutions to some nonlinear elliptic equations in $\R^2$}, Duke Math. J. \textbf{71} (1993), 427-439.


\bibitem{FH} \textsc{C. Fefferman, K. Hirachi}, \emph{Ambient metric construction of $Q$-curvature in conformal and CR geomotries}, Mathematical Research Letters \textbf{10} (2003), 819-831.

\bibitem{Gia-Mar}
\textsc{M. Giaquinta, L. Martinazzi}, 
\emph{An introduction to the regularity theory for elliptic systems, harmonic maps and minimal graphs}, Second edition, 
 [Lecture Notes. Scuola Normale Superiore di Pisa (New Series)], 11. Edizioni della Normale, 
 Pisa, 2012. xiv+366 pp. ISBN: 978-88-7642-442-7; 978-88-7642-443-4 35-02. 
\bibitem{GT}\textsc{D. Gilbarg, N.  Trudinger}, \emph{Elliptic partial differential equations of second order}, Reprint
of the 1998 edition. Classics in Mathematics, Springer-Verlag, Berlin, 2001. xiv+517 pp. ISBN: 3-
540-41160-7.
 \bibitem{gor} \textsc{E. A. Gorin}, \emph{Asymptotic properties of polynomials and algebraic functions of several variables}, Russ. Math. Surv.  \textbf{16} (1), 93-119 (1961).
 

 \bibitem{HD}
 \textsc{X. Huang, D. Ye},
 \emph{Conformal metrics in $\mathbb{R}^{n}$ with constant $Q$-curvature and arbitrary volume},
 Calc. Var. Partial Differential Equations \textbf{54} (2015), 3373-3384.


 \bibitem{H-odd}\textsc{A.  Hyder}, \emph{Existence of entire solutions to a fractional Liouville equation in $\mathbb{R}^n$}, Rend. Lincei. Mat. Appl.  \textbf{27} (2016), 1-14.
 
\bibitem{H-vol}\textsc{A. Hyder}, \emph{Conformally Euclidean  metrics on $\mathbb{R}^n$ with arbitrary  total $Q$-curvature}, Analysis \& PDE. \textbf{10} (2017), no. 3, 635-652.

\bibitem{H-clas}
\textsc{A. Hyder}, 
\emph{Structure of conformal metrics on $\mathbb{R}^n$ with constant Q-curvature}, to appear in Differential and Integral Equations (2019),  arXiv: 1504.07095 (2015).

\bibitem{HM} \textsc{A. Hyder, L. Martinazzi}, \emph{Conformal metrics on $\R^n$ with constant $Q$-curvature, prescribed volume, asymptotic behavior}, Discr. Cont. Dynamical Systems - A \textbf{35} (2015), 283-299.

\bibitem{JMMX}
\textsc{T. Jin, A. Maalaoui, L. Martinazzi, J. Xiong},
\emph{Existence and asymptotics for solutions of a non-local $Q$-curvature equation in dimension three},
 Calc. Var. Partial Differential Equations \textbf{52} (2015) no. 3-4, 469-488.


\bibitem{Lin} \textsc{C. S. Lin}, \emph{A classification of solutions of conformally invariant fourth order equations in $\R^{n}$}, Comm. Math. Helv \textbf{73} (1998), 206-231.

\bibitem{MarClass} \textsc{L. Martinazzi,}  \emph{Classification of solutions to the higher order Liouville's equation on $\mathbb{R}^{2m}$}, Math. Z. \textbf{263} (2009), 307-329.

\bibitem{LM-vol}
\textsc{L. Martinazzi},
\emph{Conformal metrics on $\mathbb{R}^{n}$ with constant $Q$-curvature and large volume},
Ann. Inst. Henri Poincar\'e (C) \textbf{30} (2013), 969-982.


\bibitem{PT} \textsc{J. Prajapat, G. Tarantello}, \emph{On a class of elliptic problems in $\R^2$: symmetry and uniqueness results}, Proc. Royal Soc. Edinburgh \textbf{131A} (2001), 967-985.

\bibitem{Tr} \textsc{M. Troyanov}, \emph{Prescribing curvature on compact surfaces with conical singularities}, Trans.  Am. Math. Soc. \textbf{324}  (1991) 793-821.

\bibitem{WX} \textsc{J. Wei, X-W. Xu}, \emph{Classification of solutions of higher order conformally invariant equations}, Math. Ann \textbf{313} (1999), 207-228.


\bibitem{WY} \textsc{J. Wei,  D. Ye}, \emph{Nonradial solutions for a conformally invariant fourth order equation in $\mathbb{R}^3$,} Calc. Var. Partial Differential Equations \textbf{32} (2008), no. 3, 373-386.
 
 \bibitem{Xu}\textsc{X. Xu}, 
\emph{Uniqueness and non-existence theorems for conformally invariant equations}, 
J. Funct. Anal. \textbf{222} (2005), 1-28.
 
 
\end{thebibliography}
\end{document}